\documentclass[12pt]{amsart}
\usepackage{amsfonts}
\usepackage{latexsym}
\usepackage{amssymb}
\usepackage{amsmath}
\usepackage{color}
\usepackage{bbm}
\usepackage{tikz}
\usepackage{enumerate}
\usepackage{mathrsfs}
\usepackage{todonotes}

\usepackage[left=2.8cm, top=3.0cm,bottom=3.0cm,right=2.8cm]{geometry}

\newcommand{\eee}{{\rm e}}  

\newcommand{\F}{\mathbb F}

\newcommand{\R}{\mathbb R}
\newcommand{\N}{\mathbb N}
\newcommand{\C}{\mathbb C}
\newcommand{\E}{\mathbb E}
\newcommand{\Pro}{\mathbb P}

\newcommand{\Sc}{\mathcal{S}}
\newcommand{\vol}{\mathrm{vol}}
\newcommand{\Mat}{\mathrm{Mat}}

\def\dint{\textup{d}}

\newcommand{\eps}{\varepsilon}

\newcommand{\toas}{\overset{{a.s.}}{\underset{n\to\infty}\longrightarrow}}

\newtheorem{thm}{Theorem} 

\newtheorem{lemma}[thm]{Lemma}

\newtheorem{proposition}[thm]{Proposition}

\theoremstyle{remark}
{
\newtheorem{rmk}[thm]{Remark}

}

\usepackage{nomencl}
\makenomenclature

\begin{document}


\title[]{Exact asymptotic volume and volume ratio\\ of Schatten unit balls}

\author[Z. Kabluchko]{Zakhar Kabluchko}
\address{Zakhar Kabluchko: Institut f\"ur Mathematische Stochastik, Westf\"alische Wilhelms-Uni\-ver\-sit\"at M\"unster, Germany}
\email{zakhar.kabluchko@uni-muenster.de}

\author[J. Prochno]{Joscha Prochno}
\address{Joscha Prochno: School of Mathematics \& Physical Sciences, University of Hull, United Kingdom} \email{j.prochno@hull.ac.uk}

\author[C. Th\"ale]{Christoph Th\"ale}
\address{Christoph Th\"ale: Fakult\"at f\"ur Mathematik, Ruhr-Universit\"at Bochum, Germany} \email{christoph.thaele@rub.de}

\keywords{Asymptotic geometric analysis, convex bodies in high dimensions, logarithmic potential theory with external field, Schatten classes, Schatten balls, Ullman distribution, volume, volume ratio}
\subjclass[2010]{Primary: 52A23. Secondary: 46B06, 60B20, 46B07, 47B10, 52A21, 31A15.}



\begin{abstract}
The unit ball $B_p^n(\R)$ of the finite-dimensional Schatten trace class $\mathcal S_p^n$ consists of all real $n\times n$ matrices $A$ whose singular values $s_1(A),\ldots,s_n(A)$ satisfy $s_1^p(A)+\ldots+s_n^p(A)\leq 1$, where $p>0$. Saint Raymond [Studia Math.\ 80, 63--75, 1984] showed that the limit
$$
\lim_{n\to\infty} n^{1/2 + 1/p} \big(\text{Vol}\, B_p^n(\R)\big)^{1/n^2}
$$
exists in $(0,\infty)$ and provided both lower and upper bounds.  In this paper we determine the precise limiting constant based on ideas from the theory of logarithmic potentials with external fields. A similar result is obtained for complex Schatten balls. As an application we compute the precise asymptotic volume ratio of the Schatten $p$-balls, as $n\to\infty$, thereby extending Saint Raymond's estimate in the case of the nuclear norm ($p=1$) to the full regime $1\leq p \leq \infty$ with exact limiting behavior.
\end{abstract}

\maketitle


\section{Introduction and main results}
\subsection{Introduction}
The Schatten trace classes $\mathcal S_p$ ($0< p \leq \infty$), consisting of all compact linear operators on a Hilbert space for which the sequence of their singular values belongs to the sequence space $\ell_p$, are one of the most important classes of unitary operator ideals. Their analysis, particularly in the finite-dimensional setting, has a long tradition in asymptotic geometric analysis and the local theory of Banach spaces. For example, Gordon and Lewis \cite{GL1974} obtained that the space $\mathcal S_1$ does not have local unconditional  structure, Tomzcak-Jaegermann \cite{TJ1974} demonstrated that this space (which is naturally identified with the projective tensor product $\ell_2\otimes_{\pi}\ell_2$) has Rademacher cotype $2$, and K\"onig, Meyer and Pajor \cite{KoenigMeyerPajor} proved the boundedness of the isotropic constants of $\mathcal S_p^n$ ($1\leq p \leq \infty$). More recently, Gu\'edon and Paouris \cite{GuedonPaourisConcentration} have established concentration of mass properties for the unit balls of Schatten $p$-classes $\mathcal S_p$, Barthe and Cordero-Erausquin \cite{BartheCordero-Erausquin} studied variance estimates, Radke and Vritsiou \cite{RadkeVritsiou} proved the thin-shell conjecture, and Hinrichs, Prochno and Vyb\'iral \cite{HPV17} computed the entropy numbers for their natural embeddings.

\subsection{Asymptotic volume of Schatten balls}
In \cite{R1984}, Saint Raymond studied the volumetric properties of unit balls in finite-dimensional real and complex Schatten $p$-classes. For $0<p\le \infty$, let $\mathcal S_p^n(\F)$ denote the space of all $n \times n$ matrices $A$ with entries from $\F\in\{\R,\C\}$ equipped with the Schatten $p$-(quasi)norm
\[
\|A\|_{\Sc^n_p}
=\begin{cases}
\bigg(\sum\limits_{j=1}^n s_j(A)^p\bigg)^{1/p} &: p<\infty\\
\max\big\{s_1(A),\ldots,s_n(A)\big\} &: p=\infty,
\end{cases}
\]
where $s_1(A),\ldots,s_n(A)$ are the singular values of $A$. If we denote by
$$
B_p^n(\F) = \{A\in\mathcal{S}_p^n:\|A\|_{\mathcal{S}_p^n}\leq 1\}
$$
the corresponding Schatten unit ball, Saint Raymond proved asymptotic formulas for their volume, showing that, as $n\to\infty$,
\begin{align}\label{eq:RaymondReal}
\big(\text{Vol}_{n^2}\, B_p^n(\R)\big)^{1/n^2} &\sim n^{-\frac 12 -\frac 1p} \sqrt{2\pi \eee^{3/2} \Delta(p/2)}
\end{align}
and
\begin{align}\label{eq:RaymondComplex}
\big(\text{Vol}_{2n^2}\, B_p^n(\C)\big)^{1/(2n^2)} &\sim n^{-\frac 12 -\frac 1p} \sqrt{\pi \eee^{3/2} \Delta(p/2)}\,,
\end{align}
where $\Delta(p)\in(0,\infty)$ is certain constant and $\text{Vol}_N$ stands for the Lebesgue measure of dimension $N\in\N$. Here and below we shall write $a_n\sim b_n$ for two sequences $(a_n)_{n\in\N}$ and $(b_n)_{n\in\N}$ whenever $a_n/b_n\to 1$, as $n\to\infty$. For the parameter $\Delta(p)$ that appears in \eqref{eq:RaymondReal} and \eqref{eq:RaymondComplex} he provided both lower and upper bounds. However, with the exception of $\Delta(1)=\eee^{-1/2}$ and $\Delta(\infty)=1/4$ no explicit values of $\Delta(p)$ seem to be known.
With this paper we want to shed light on the precise value of $\Delta(p)$ and the asymptotic volume of the unit balls in finite-dimensional Schatten $p$-classes for all $0<p\leq \infty$. Concomitantly, we shall study another important quantity related to the geometry of Banach spaces, the (asymptotic) volume ratio of $\mathcal S_p^n(\F)$; see below. Our main result is the explicit computation of $\Delta(p)$.

\begin{thm}\label{theo:main_Delta_p}
For all $0<p\leq \infty$  we have
$$
\Delta(p)= \frac 14 \left(\frac{2\sqrt \pi\, \Gamma(p+1)}{\sqrt \eee\, \Gamma(p+\frac 12)}\right)^{1/p}.
$$
\end{thm}

The proof relies on a reduction trick, turning the discrete variational problem of \cite{R1984} into a continuous one, which allows us to use the theory of logarithmic potentials with external fields. In some part of a companion paper~\cite{KPTMatrixBalls}, we study a similar question for unit balls in the spaces of self-adjoint matrices. There are some similarities between both problems, although it seems that neither one can be reduced to the other. The main difference is that Schatten $p$-balls lead to variational problems on the positive half-line, whereas self-adjoint balls lead to variational problems on the whole line.

\subsection{Volume ratio of Schatten balls}
We turn now to an application of Theorem \ref{theo:main_Delta_p}. For this let us recall that for a $N$-dimensional convex body $K$ the \textit{volume ratio} ${\rm vr}(K)$ is defined as
$$
{\rm vr}(K) = \inf\left({\text{Vol}_N(K)\over\text{Vol}_N(\mathcal{E})}\right)^{1/N},
$$
where the infimum is taken over all ellipsoids $\mathcal{E}$ which are contained in $K$. One can in fact show that there is a unique ellipsoid $\mathcal{E}$, referred to as the John ellipsoid, of maximal volume which is contained in $K$. The volume ratio is a very powerful concept in asymptotic geometric analysis that has its origin in the groundbreaking works of Szarek \cite{S1978}, and Szarek and Tomczak-Jaegermann \cite{ST1980} who extracted this core notion behind a famous result of Ka\v{s}in on nearly Euclidean decompositions of $\ell_1^n$ and successfully applied their ideas to study several classes of finite-dimensional normed spaces using this affine invariant.
Since then the volume ratio appeared in many places, for example, in an estimate of the volume ratio in terms of the Rademacher cotype-$2$ constant by Bourgain and Milman \cite{BM1987}, in Ball's volume ratio inequality \cite{BallInequality}, or in Bourgain, Klartag and Milman's reduction of the hyperplane conjecture \cite{BourgainKlartagMilman}. It has also significant applications in approximation theory. We refer the reader to the monographs \cite{AsymptoticGeometricAnalysisBookPart1,IsotropicConvexBodies,TJ1974} for further background material. Before we state our next result let us recall that Szarek and Tomczak-Jaegermann \cite[Proposition 3.1]{ST1980} proved that, for all $n\in\N$,
\[
{\rm vr}\big(B_1^n(\F)\big) \leq 32\,000,
\]
at the same time obtaining bounds for general unitary operator ideals of operators on Hilbert spaces \cite[Proposition 3.2]{ST1980}.
Later, Saint Raymond \cite[Th\'eor\`eme 9]{R1984} obtained asymptotically, as $n\to\infty$,
\[
{\rm vr}\big(B_1^n(\R)\big) \sim \sqrt{\frac{\Delta(1/2)}{\Delta(1)}} \leq \frac{2}{e^{1/4}}.
\]
Our main result, Theorem \ref{theo:main_Delta_p}, can be used to determine asymptotically, as $n\to\infty$, the precise volume ratio of the Schatten $p$-balls $B_p^n(\F)$ in the full regime $1\leq p \leq \infty$.

\begin{thm}\label{thm:VolumeRatio}
Let $\F\in\{\R,\C\}$. For $1\leq p<2$ we have that, as $n\to\infty$,
$$
{\rm vr}\big(B_p^n(\F)\big) \sim \sqrt{\Delta(p/2)\over\Delta(1)} =  {1\over 2}\left({\Gamma({p\over 2}+1)\over\Gamma({p\over 2}+{1\over 2})}\right)^{1\over p}\sqrt{\eee^{{1\over 2}-{1\over p}}(4\pi)^{1\over p}}\,,
$$
while if $2\leq p\leq\infty$ we have, as $n\to\infty$,
$$
{\rm vr}\big(B_p^n(\F)\big) \sim n^{{1\over 2}-{1\over p}}\sqrt{\Delta(p/2)\over\Delta(1)} =  n^{{1\over 2}-{1\over p}}{1\over 2}\left({\Gamma({p\over 2}+1)\over\Gamma({p\over 2}+{1\over 2})}\right)^{1\over p}\sqrt{\eee^{{1\over 2}-{1\over p}}(4\pi)^{1\over p}}\,.
$$
\end{thm}


\subsection{Organization of the paper}
In Section \ref{subsec:Reduction} we present the reduction trick of Saint Raymond's discrete variational problem to a corresponding continuous problem that allows us to apply methods and ideas from the theory of logarithmic potentials with external fields. Section \ref{subsec:Ullmann} is devoted to the Ullman distribution, which is the unique maximizer in this variational problem, while in Section \ref{subsec:Proof} we present the proof of Theorem \ref{theo:main_Delta_p}. Finally, in Section \ref{subsec:Proof2}, we prove Theorem \ref{thm:VolumeRatio}.

\section{Proof of Theorem~\ref{theo:main_Delta_p}}

We start with the following observation.
The equality $\Delta(\infty) = 1/4$ was established by Saint Raymond~\cite[Corollaire~4 on p.~69]{R1984}.
For $0<p<\infty$, which is always assumed in the following, Saint Raymond~\cite[p.~70]{R1984} characterized the constant $\Delta(p)$ as the limit of $\Delta_n(p)$, as $n\to\infty$, where
\begin{equation}\label{eq:log_Delta_n}
\log \Delta_n (p) = \sup_{0\leq t_1\leq \ldots \leq t_n} \left( \frac{2}{n(n-1)} \sum_{1\leq i< j\leq n}\log |t_i-t_j| - \frac 1p \log \left(\frac1n  \sum_{i=1}^n t_i^p\right) \right).
\end{equation}
He then showed that the positive sequence $\Delta_n(p)$ is decreasing, which implies that it converges to a limit, as $n\to\infty$. By providing bounds on this limit, he showed that it is non-zero, but, as \eqref{eq:RaymondReal} and \eqref{eq:RaymondComplex} show, these bounds were not sharp. We shall compute the limiting constant precisely.

\subsection{Reduction to the continuous problem}\label{subsec:Reduction}

The first step in the computation of $\Delta(p)$ is to replace the supremum over the points $0\leq t_1\leq \ldots\leq t_n$ in \eqref{eq:log_Delta_n} by its continuous version, namely the supremum over probability measures on the positive half-line. For this purpose, let $\mathcal M_1^p(\R)$ be the set of all probability measures on $\R$ with $\int_{\R} |x|^p \mu(\dint x) < \infty$. Similarly, denote by $\mathcal M_1^p(\R_+)$ the set of all probability measures on $\R_+ = [0,\infty)$ that satisfy $\int_{\R_+} x^p \mu(\dint x) < \infty$. Let us write $\delta_0$ for the Dirac measure at $0$.

On the set $\mathcal M_1^p(\R)\backslash \{\delta_0\}$ we consider the functional
\begin{equation}
\mathscr J_p(\mu):= \int_{\R} \int_{\R} \log|x-y|\,\mu(\dint x)\,\mu(\dint y) - \frac{1}{p}\log \int_\R |x|^p \,\mu(\dint x),
\end{equation}
which takes values in $\R\cup\{-\infty\}$. We shall now demonstrate that the limit of $(\log \Delta_n(p))_{n\in\N}$ coincides with the supremum that $\mathscr J_p(\cdot)$ takes on the set $\mathcal M_1^p(\R_+)\backslash \{\delta_0\}$.

\begin{proposition}\label{prop:cont_sup}
For all $0<p<\infty$ we have
$$
\lim_{n\to\infty}\log \Delta_n(p) = \sup_{\mu \in \mathcal M_1^p(\R_+)\backslash\{\delta_0\}}  \mathscr J_p(\mu).
$$
\end{proposition}
\begin{proof}
We split the proof into a lower and an upper bound.

\medbreak
\paragraph{\it Lower bound.} Let us prove that
\begin{equation}\label{eq:lower_bound}
\lim_{n\to\infty} \log \Delta_n(p) \geq \mathscr J_p(\mu)
\end{equation}
for an arbitrary probability measure $\mu \in \mathcal M_1^p(\R_+)\backslash \{\delta_0\}$. We assume that
$$
\int_{\R_+} \int_{\R_+} \log|x-y|\,\mu(\dint x)\,\mu(\dint y) \neq -\infty,
$$
because otherwise the statement is evident.
Since by the definition of $\mathcal{M}_1^p(\R_+)\backslash\{\delta_0\}$ the $p$-th moment of $\mu$ is finite, the above double integral cannot take the value $+\infty$, so it must be finite.  Let $V_1,V_2,\ldots$ be i.i.d.\ non-negative random variables with probability distribution $\mu$. The strong law of large numbers for $U$-statistics, see~\cite[Theorem 3.1.1]{koroljuk_book}, yields the almost sure convergence
\begin{equation}\label{eq:ASC1}
\frac 2 {n(n-1)} \sum_{1\leq i < j\leq n} \log |V_i-V_j| \toas \E \log |V_1-V_2|.
\end{equation}
On the other hand, the strong law of large numbers for sums of i.i.d.\ random variables yields
\begin{equation}\label{eq:ASC2}
\frac 1p \log \left(\frac 1n \sum_{i=1}^n V_i^p\right) \toas \frac 1p \log \E V_1^{p}.
\end{equation}
By~\eqref{eq:log_Delta_n}, we have that for each realization
\[
\log \Delta_n (p)
\geq
\frac 2 {n(n-1)} \sum_{1\leq i < j\leq n} \log |V_i-V_j| - \frac 1p \log \left(\frac 1n \sum_{i=1}^n V_i^p\right).
\]
Therefore, using~\eqref{eq:ASC1} and \eqref{eq:ASC2}, we get
\[
 \log \Delta(p) = \lim_{n\to\infty} \log \Delta_n(p) \geq \E \log|V_1-V_2|-\frac{1}{p}\log \E V_1^p = \mathscr J_{p}(\mu),
\]
thus proving the lower bound~\eqref{eq:lower_bound}.

\medbreak

\paragraph{\it Upper bound.} Our next aim is to prove that there is a sequence of probability measures $\nu_1,\nu_2,\ldots \in \mathcal M_1^p(\R_+)\backslash\{\delta_0\}$ such that
\begin{equation}\label{eq:upper_bound}
\liminf_{n\to\infty} \mathscr J_p(\nu_n)\geq \log \Delta(p).
\end{equation}
Saint Raymond~\cite[Lemme~6 on p.~71]{R1984} showed that there is a maximizer of the right-hand side of~\eqref{eq:log_Delta_n}, which we denote by $(t_{1,n}^*,\dots,t_{n,n}^*)$, and which has the following properties:
\begin{equation}\label{eq:raymond_properties}
0 = t_{1,n}^* < \ldots < t_{n,n}^*,
\qquad
\frac 1n \sum_{i=1}^n (t_{i,n}^*)^p = 1,
\qquad\text{and}\qquad
t_{i,n}^*-t_{i-1,n}^* \geq n^{-C}
\end{equation}
for all $n\geq 2$,  $i\in \{2,\ldots,n\}$ and some constant $C=C(p)>0$.  In fact, Saint Raymond normalized the $\ell_p$-norm of the maximizer to be $1$, but for us it is more convenient to set it to be $n^{1/p}$, as above. This is possible since the expression on the right-hand side of~\eqref{eq:log_Delta_n} remains unchanged if we replace $(t_1,\ldots,t_n)$ by $(a t_1,\ldots, at_n)$ for $a>0$.  For  $n\geq 2$, we put $\eps_n := n^{-2 C}$ and consider an absolutely continuous probability measure $\nu_n$ on $\R_+$ with Lebesgue density
$$
f_n(t) =
\frac 1 {n\eps_n} \mathbbm 1_{[0, \eps_n]} (t) +
\frac 1 {n\eps_n} \sum_{i=2}^n \mathbbm 1_{[t_{i,n}^*-\eps_n, t_{i,n}^*]}(t),\qquad t\in\R_+.
$$
Note that $\nu_n$ is the uniform distribution on the union of the intervals $B_{1,n}:=[0, \eps_n]$ and $B_{i,n} := [t_{i,n}^*-\eps_n, t_{i,n}^*]$, for $i=2,\ldots,n$.
For sufficiently large $n$, the intervals $B_{1,n},\ldots,B_{n,n}$ are disjoint by~\eqref{eq:raymond_properties} and we have
\begin{align}\label{ineq:integral estimate}
\int_{\R_+} t^p f_n(t) \,\dint t = \frac 1 {n\eps_n} \sum_{i=1}^n  \int_{B_{i,n}} t^p\, \dint t
\leq \frac{\eps_n^p}{n}  + \frac 1n \sum_{i=2}^n (t_{i,n}^*)^p
= \frac{\eps_n^p}{n} + \frac 1n \sum_{i=1}^n (t_{i,n}^*)^p.
\end{align}
We claim that
\begin{equation}\label{eq:log_energy_f_n}
\int_{\R_+}\int_{\R_+} f_n(x) f_n(y) \log |x-y| \,\dint x \,\dint y
= \frac{2}{n^2} \sum_{1\leq i < j \leq n} \log |t_{i,n}^* - t_{j,n}^*| -o(1),
\end{equation}
where $o(1)$ denotes any sequence converging to $0$ as $n\to\infty$. From this in combination with \eqref{ineq:integral estimate} it would follow that
\begin{align*}
&\liminf_{n\to\infty} \mathscr J_p(\nu_n)\\
&=
\liminf_{n\to\infty} \left(\int_{\R_+}\int_{\R_+} f_n(x) f_n(y) \log |x-y| \,\dint x\, \dint y - \frac 1p \log \left(\int_{\R_+} t^p f_n(t)\,\dint t\right)\right)
\\
&\geq
\liminf_{n\to\infty}
\left(\frac{2}{n^2} \sum_{1\leq i < j \leq n} \log |t_{i,n}^* - t_{j,n}^*| -o(1) - \frac{1}{p}\log \left(\frac{1}{n}\sum_{i=1}^n (t_{i,n}^*)^p +\frac{\varepsilon_n^p}{n}\right)\right)\\
&=
\liminf_{n\to\infty}
\left(\frac{2}{n^2} \sum_{1\leq i < j \leq n} \log |t_{i,n}^* - t_{j,n}^*| -o(1) - \frac{1}{p}\log \left(\frac{1}{n}\sum_{i=1}^n (t_{i,n}^*)^p\right)\right)\\
&=
\liminf_{n\to\infty}
\left(\frac{2}{n(n-1)} \sum_{1\leq i < j \leq n} \log |t_{i,n}^* - t_{j,n}^*| - \frac 1p \log \left(\frac{1}{n}\sum_{i=1}^n (t_{i,n}^*)^p\right)\right)\frac{n(n-1)}{n^2}\\
&= \liminf_{n\to\infty} \log \Delta_n(p) =
\log \Delta(p),
\end{align*}
thus proving the upper bound~\eqref{eq:upper_bound}. Here, we used that $\frac 1n \sum_{i=1}^n (t_{i,n}^*)^p = 1$ and $\eps_n\to 0$ as $n\to\infty$.

%

\medbreak

\paragraph{\it Proof of~\eqref{eq:log_energy_f_n}.} We represent the double integral on the left-hand side of~\eqref{eq:log_energy_f_n} as the following double sum,
$$
\int_{\R_+}\int_{\R_+} f_n(x) f_n(y) \log |x-y| \,\dint x\, \dint y =  \frac {1}{n^2\eps_n^2}\sum_{i=1}^n\sum_{j=1}^n \int_{B_{i,n}}\int_{B_{j,n}} \log |x-y| \,\dint x \,\dint y.
$$
Observe that each summand on the right-hand side represents the ``interaction'' between the intervals $B_{i,n}$ and $B_{j,n}$ for $i,j\in\{1,\dots,n\}$.

\medbreak

\paragraph{\it Case 1: Self-interaction terms.} Let us take some $i\in \{2,\ldots,n\}$ and consider  the interaction of the interval $B_{i,n} = [t_{i,n}^* -\eps_n, t_{i,n}^*]$ with itself. If we denote by $X$ and $Y$ two independent random variables with uniform distribution on the interval $[0,1]$, then $t_{i,n}^* -\eps_n X$ and $t_{i,n}^* -\eps_n Y$ are uniformly distributed on the interval $B_{i,n}$ and we can write
\begin{align*}
\frac 1 {n^2 \eps_n^2} \int_{B_{i,n}} \int_{B_{i,n}} \log |x-y| \,\dint x\, \dint y
&=
\frac {\E \log |(t_{i,n}^* -\eps_n X) - (t_{i,n}^* -\eps_n Y)|}{n^2}\\
&=
\frac {\log \eps_n + \E \log |X-Y|}{n^2}
=
O\left(\frac{\log n} {n^2}\right)
\end{align*}
by the choice of $\eps_n$. Here we write $a_n=O(b_n)$ for two sequences $(a_n)_{n\in\N}$ and $(b_n)_{n\in\N}$ if there exists a constant $M\in(0,\infty)$ such that $|a_n|\leq Mb_n$ for all sufficiently large $n$. An analogous estimate also holds for  $i=1$.  For the sum of self-interaction terms we thus obtain the upper bound
$$
\sum_{i=1}^n \frac 1 {n^2 \eps_n^2} \int_{B_{i,n}} \int_{B_{i,n}} \log |x-y| \,\dint x\, \dint y = O\left(\frac{\log n} {n}\right)=o(1).
$$

\medbreak

\paragraph{\it Case 2: Interactions between different intervals.} Take some  $i,j\in \{2,\ldots,n\}$ with $i\neq j$.  If $X$ and $Y$ are, as above, independent random variables with uniform distribution on the interval $[0,1]$, then the random variables $t^*_{i,n} - \eps_n X$ and $t^*_{j,n} - \eps_n Y$ are uniformly distributed  on the intervals  $B_{i,n}= [t_{i,n}^*-\eps_n, t_{i,n}^*]$ and $B_{j,n} = [t_{j,n}^*-\eps_n, t_{j,n}^*]$, respectively.  Thus,
\begin{align*}
\frac 1 {n^2 \eps_n^2} \int_{B_{i,n}} \int_{B_{j,n}} \log |x-y| \,\dint x\, \dint y
&=
\frac {\E \log |(t_{i,n}^* - \eps_n X) - (t_{j,n}^* - \eps_n Y)|}{n^2} \\
&=
\frac 1 {n^2} \, \E \log |t_{i,n}^*  - t_{j,n}^*| +  \frac 1{n^2} \E \log \left|1 +  \eps_n  \frac{Y- X}{t_{i,n}^*  - t_{j,n}^*} \right|.
\end{align*}
Recalling that $|t_{i,n}^*-t_{j,n}^*| > n^{-C}$ and $\eps_n = n^{-2C}$, we arrive at
$$
\frac 1{n^2}\, \E \log \left|1 +  \eps_n  \frac{Y-X}{t_{i,n}^*  - t_{j,n}^*} \right|
=
\frac 1{n^2}\,  O\left(  \frac{\eps_n}{|t_{i,n}^*  - t_{j,n}^*|} \right)
=
O\left(\frac1{n^{2+C}}\right).
$$
The same estimate applies if $i=1$ and $j\in \{2,\ldots,n\}$, but this time the uniform distribution on the intervals $B_{1,n}=[0,\eps_n]$ and $B_{j,n} = [t_{j,n}^*-\eps_n , t_{j,n}^*]$ is represented by the random variables $\eps_n X$ and $t_{j,n}^* - \eps Y_n$.

\medbreak

Taking together the estimates of Case 1 and Case 2, we arrive at
$$
\int_{\R_+}\int_{\R_+} f_n(x) f_n(y) \log |x-y| \,\dint x\, \dint y  =
\frac{2}{n^2} \sum_{1\leq i < j \leq n} \log |t_{i,n}^* - t_{j,n}^*| -o(1),
$$
which completes the proof of~\eqref{eq:log_energy_f_n}.
\end{proof}

\subsection{The Ullman distribution: Maximizer of the functional $\mathscr J_p$}\label{subsec:Ullmann}
In view of Proposition~\ref{prop:cont_sup} it remains to compute the supremum of the functional $\mathscr J_p(\mu)$ over $\mu\in \mathcal M_1^p (\R_+) \backslash \{\delta_0\}$. In fact, the maximizer of the same functional over the larger space $\mathcal M_1^p (\R) \backslash \{\delta_0\}$ is known to be the so-called Ullman distribution.

Let $0< p < \infty$.
We say that a random variable $\mathbb U$ which takes values in the interval  $[-1,1]$ has \textit{Ullman distribution} with parameter $p$, and write $\mathbb U\sim\mathscr {U}(p)$, if its Lebesgue density is given by
\begin{equation}\label{eq:ullman_def}
h_p(x) := {p\over\pi} \int_{|x|}^1{t^{p-1}\over\sqrt{t^2-x^2}}\,\dint t, \qquad x\in [-1,1].
\end{equation}
The Ullman distribution appears as the equilibrium distribution for electric charges on the real line in the external field of the form of a constant multiple of $|x|^p$.  Namely, for a probability measure $\mu\in \mathcal M_1^p(\R)\backslash\{\delta_0\}$ consider the energy functional
$$
\mathscr E_p(\mu) := \int_{\R} \int_{\R} \log \frac{1}{|x-y|}\, \mu(\dint x)\, \mu(\dint y) + 2 \int_{\R} Q_p(x) \mu(\dint x),
$$
with the external field
$$
Q_p(x) = \frac{\sqrt \pi \, \Gamma(\frac{p}{2})} {2\Gamma(\frac{p+1}{2})} |x|^p, \quad x\in\R.
$$
Then the unique minimizer of $\mathscr E_p$ is the Ullman distribution on the interval $[-1,1]$ with density $h_p$ (see, e.g., \cite[Theorem~5.1 on p.~240]{SaffBOOK}). As an easy consequence,  one can derive the following proposition (see \cite[Lemma 3.6]{KPTMatrixBalls}).

\begin{proposition}\label{prop:maximizers_over_R}
Let $p>0$. The only maximizers of the functional
$$
\mathscr J_p(\mu)= \int_{\R} \int_{\R} \log|x-y|\,\mu(\dint x)\,\mu(\dint y) - \frac{1}{p}\log \int_\R |x|^p \,\mu(\dint x)
$$
over $\mathcal M_1^p (\R)\backslash\{\delta_0\}$ are probability measures with densities $\frac{1}{b}h_p(\frac{x}{b})$, $b>0$, where $h_p$ is the Ullman density~\eqref{eq:ullman_def}.
\end{proposition}

In the following  we shall also need two more properties of the Ullman distribution, which can be verified by direct computation; see, e.g., \cite[Section 2.5]{KPTMatrixBalls}.

\begin{lemma}\label{lem:moment ullman}
Let $p>0$ and let $U\sim \mathscr U(p)$ and $V\sim \mathscr U(p)$ be two independent Ullman random variables. Then
\begin{align*}
\E|U|^p = \int_{-1}^1 h_p(x) |x|^p\, \dint x =  \frac{\Gamma(\frac{p+1}{2})}{2
\sqrt{\pi}\,\Gamma(\frac{p+2}{2})}
\end{align*}
and
\begin{align*}
\E \log|U-V| = \int_{-1}^1 \int_{-1}^1 h_p(x)\,h_p(y)\,\log |x-y|\, \dint x\, \dint y = -\log 2 - \frac{1}{2p}.
\end{align*}
\end{lemma}

Finally, we are able to maximize $\mathscr J_p(\mu)$ over $\mathcal M_1^p(\R_+)\backslash\{\delta_0\}$.
\begin{proposition}\label{prop:maximizer_is_ullman}
For each $p>0$ we have
$$
\sup_{\mu \in \mathcal M_1^p(\R_+)\backslash\{\delta_0\}}
\mathscr J_p(\mu)
=
-2\log 2  + \frac 1p \log \left(\frac{2\sqrt \pi\, \Gamma(p+1)}{\sqrt \eee\,\Gamma(p+\frac 12)}\right).
$$
\end{proposition}
\begin{proof}
We reduce the problem on the half-line to the problem on the whole line by a trick known in the theory of orthogonal polynomials~\cite[\S 6]{rakhmanov}.
Let $V$ be a random variable with distribution $\mu\in \mathcal M_1^p(\R_+)\backslash\{\delta_0\}$ and denote by $\widetilde V$ an independent copy of $V$. Then we are interested in maximizing the expression
$$
\mathscr J_p(\mu) = \E \log |V-\widetilde V| - \frac 1p \log \E V^p
$$
over all possible choices for $V\geq 0$ with $\E V^p < \infty$ and $\Pro[V=0]<1$. Independently of $V$, consider a symmetric Rademacher random variable $\eps$ with $\Pro[\eps=1]=\Pro[\eps=-1] = 1/2$ and define $U:=\eps \sqrt V$. Then $U$ has the same distribution as $-U$, and $U^2=V$. Let also $\widetilde U$ be an independent copy of $U$. With this notation, we can write
\begin{align*}
\mathscr J_p(\mu)
&=
\E \log |V-\widetilde V| - \frac 1p \log \E V^p \\
&=
\E \log |U^2 - \widetilde U^2| - \frac 1p \log \E |U|^{2p}\\
&=
\E \log |U - \widetilde U| + \E \log |U+\widetilde U| - \frac 1p \log \E |U|^{2p}\\
&=
2 \left(\E \log |U - \widetilde U| - \frac 1{2p} \log \E |U|^{2p}\right)\\
&=
2\mathscr J_{2p}(\mathcal L_U),
\end{align*}
where in the penultimate line we used that $U+ \widetilde U$ has the same distribution as $U-\widetilde U$, and where $\mathcal L_U$ is the probability distribution of $U$. Since $\mathcal L_U \in  \mathcal M_1^{2p} (\R)\backslash \{\delta_0\}$, Proposition~\ref{prop:maximizers_over_R} with $p$ replaced by $2p$ yields that \begin{align*}
\mathscr J_p(\mu)
=
2\mathscr J_{2p}(\mathcal L_U) \leq
2\int_{\R} \int_{\R} \log|x-y|\,h_{2p}(x) h_{2p}(y) \,\dint x\, \dint y - \frac{1}{p}\log \int_\R |x|^{2p}  h_{2p} (x) \,\dint x.
\end{align*}
Moreover, if $U$ would have Lebesgue density $h_{2p}$, then the previous inequality would turn into an equality. The right-hand side above can be computed explicitly using Lemma \ref{lem:moment ullman}. In fact,
\begin{align*}
\mathscr J_p(\mu) \leq
-2\log 2 - \frac 1 {2p} - \frac 1 {p} \log \frac{\Gamma(p+\frac 12)}{2
\sqrt{\pi}\,\Gamma(p+1)}
=-2\log 2  + \frac 1p \log \left(\frac{2\sqrt \pi\, \Gamma(p+1)}{\sqrt \eee\,\Gamma(p+\frac 12)}\right)
\end{align*}
with equality if $\mu$ is the distribution of $\sqrt  {|U|}$, where $U\sim\mathscr{U}(2p)$ is Ullman distributed with parameter $2p$.
\end{proof}

\subsection{Proof of Theorem~\ref{theo:main_Delta_p}}\label{subsec:Proof}
By combining Proposition \ref{prop:cont_sup} with Proposition \ref{prop:maximizer_is_ullman} we obtain
$$
\log \Delta(p) = \lim_{n\to\infty} \log \Delta_n(p) = \sup_{\mu \in \mathcal M_1^p(\R_+)\backslash\{\delta_0\}}
\mathscr J_p(\mu)
=
-2\log 2  + \frac 1p \log \left(\frac{2\sqrt \pi\, \Gamma(p+1)}{\sqrt \eee\,\Gamma(p+\frac 12)}\right).
$$
By exponentiating, we arrive at the required formula for $\Delta(p)$. \hfill $\Box$.

\section{Proof of Theorem~\ref{thm:VolumeRatio}}\label{subsec:Proof2}

Let us recall some definitions and provide some additional preliminaries. Let $X$ be a real $N$-dimensional Banach space with unit ball $B_X$. If we are given a complex Banach space, we ignore the complex structure and consider the space as a real one, so that $N$ is the dimension over $\R$. We denote by $\mathcal{E}_X$ the (unique) maximal volume ellipsoid that is contained in $B_X$. The volume ratio of $X$ is then defined as
\begin{equation}\label{eq:DefVolumeRatio}
{\rm vr}(X) = \left(\frac{\text{Vol}_N(B_X)}{\text{Vol}_N(\mathcal{E}_X)}\right)^{1/N},
\end{equation}
where $\text{Vol}_N(\,\cdot\,)$ stands for the usual $N$-dimensional Lebesgue measure. Note that if $K$ is an $N$-dimensional symmetric convex body, $K$ is the unit ball of an $N$-dimensional Banach space $X_K$ and ${\rm vr}(X_K)$ coincides with the definition of the volume ratio presented in the introduction. Let us recall from \cite[Section 16]{TJ1988} that a Banach space $X$ is said to have \textit{enough symmetries} if the only operators that commute with every isometry of $X$ are multiples of the identity. If $X$ is $N$-dimensional and has enough symmetries, it is known that $\mathcal{E}_X$ is a suitable multiple of the Euclidean unit ball of the same dimension. More precisely,
\begin{equation}\label{eq:JohnEllipsoid}
\mathcal{E}_X = \big\|{\rm id}:\ell_2^N\to X\big\|^{-1}\mathbb{B}_2^N,
\end{equation}
where $\ell_2^N$ is the $N$-dimensional Euclidean space with the Euclidean unit ball $\mathbb{B}_2^N$ and ${\rm id}:\ell_2^N\to X$ stands for the identity operator from $\ell_2^N$ to $X$ with the standard operator norm $\|{\rm id}:\ell_2^N\to X\|$. We also recall from \cite{DefantMichels} that the Schatten classes $\mathcal{S}_p^n(\F)$, where $\F\in\{\R,\C\}$, is in fact a Banach space with enough symmetries. In what follows, for $\F\in\{\R,\C\}$ we denote by $\Mat_n(\F)$ the set of all $n\times n$ matrices with entries from $\F$.

\begin{proof}[Proof of Theorem \ref{thm:VolumeRatio}]
According to what has been said above, we need to compute the operator norm
$$
\|{\rm id}:\mathcal{S}_2^n(\F)\to\mathcal{S}_p^n(\F)\|,
$$
where we used the fact that the Schatten $2$-ball $B_2^n(\F)$ is just the Euclidean unit ball of the appropriate dimension (namely $n^2$ if $\F=\R$ and $2n^2$ if $\F=\C$). We first observe that
\begin{align*}
\|{\rm id}:\mathcal{S}_2^n(\F)\to\mathcal{S}_p^n(\F)\| &
= \sup_{\|A\|_{\mathcal S_2^n}\leq 1} \|A\|_{\mathcal S_p^n}
\leq\begin{cases}
n^{{1\over p}-{1\over 2}} &: 1\leq p< 2\\
1 &: 2\leq p\leq\infty,
\end{cases}
\end{align*}
since for $1\leq p<2$,
$$
\|A\|_{\mathcal S_p^n} = \left(\sum_{i=1}^n s_i(A)^p \right)^{1/p} \leq n^{\frac 1p - \frac 12} \left(\sum_{i=1}^n s_i(A)^2 \right)^{1/2}  = n^{\frac 1p - \frac 12} \|A\|_{\mathcal S_2^n}
$$
by the inequality between the generalized means.
On the other hand, let $A_1\in\Mat_n(\F)$ be $n^{-1/2}$ times the $n\times n$ identity matrix, which has singular values $s_1(A_1),\ldots,s_n(A_1)$ all equal to $n^{-1/2}$. This shows that, if $1\leq p<2$,
$$
\|{\rm id}:\mathcal{S}_2^n(\F)\to\mathcal{S}_p^n(\F)\| \geq \left(\sum_{j=1}^nn^{-p/2}\right)^{1/p} = n^{{1\over p}-{1\over 2}}.
$$
Also, if $2\leq p\leq\infty$ we take $A_2=(a_{ij})\in\Mat_n(\F)$ to be the $n\times n$ with all entries equal to $0$, except for setting $a_{11}=1$. In this case $s_1(A_2)=1$ and $s_2(A_2)=\ldots=s_n(A_2)=0$ and so
$$
\|{\rm id}:\mathcal{S}_2^n(\F)\to\mathcal{S}_p^n(\F)\| \geq 1.
$$
Let $\beta=1$ if $\F=\R$ and $\beta=2$ if $\F=\C$. Then, taking together the upper and lower bound and plugging this into \eqref{eq:DefVolumeRatio} and \eqref{eq:JohnEllipsoid}, we conclude from \eqref{eq:RaymondReal} and \eqref{eq:RaymondComplex} that
\begin{align*}
{\rm vr}(B_p^n(\F)) &= \begin{cases}
n^{{1\over p}-{1\over 2}}\left({\vol(B_p^n(\F))\over\vol(B_2^n(\F))}\right)^{1\over \beta n^2} &: 1\leq p<2\\
\left({\vol(B_p^n(\F))\over\vol(B_2^n(\F))}\right)^{1\over \beta n^2} &: 2\leq p\leq\infty
\end{cases}\sim\begin{cases}
n^{{1\over p}-{1\over 2}}{n^{-{1\over 2}-{1\over p}}\sqrt{\Delta(p/2)}\over n^{-{1\over 2}-{1\over 2}}\sqrt{\Delta(1)}} &: 1\leq p<2\\
{n^{-{1\over 2}-{1\over p}}\sqrt{\Delta(p/2)}\over n^{-{1\over 2}-{1\over 2}}\sqrt{\Delta(1)}} &: 2\leq p\leq\infty
\end{cases}\\
&=\begin{cases}
\sqrt{\Delta(p/2)\over\Delta(1)} &: 1\leq p<2\\
n^{{1\over 2}-{1\over p}}\sqrt{\Delta(p/2)\over\Delta(1)} &: 2\leq p\leq\infty.
\end{cases}
\end{align*}
Applying now Theorem \ref{theo:main_Delta_p} and simplifying the resulting expression completes the proof.
\end{proof}

\begin{rmk}
Theorem~\ref{thm:VolumeRatio} implies that $\sup_{n\in\N} {\rm vr}(B_p^n(\F))$ is finite for $1\leq p\leq 2$. In fact, it is even possible to give an explicit upper bound on this quantity.   Indeed, since the John ellipsoid is just a rescaled Euclidean ball, its volume can be computed exactly. It remains to provide an explicit upper bound on the volume of $B_p^n(\F)$. Using the estimate in~\cite[p.~73]{R1984}, it suffices to provide an explicit upper bound on $\Delta_n(p/2)$.
To this end, one can estimate the error terms in the proof of the upper bound of Proposition~\ref{prop:cont_sup}. We refrain from providing the details.
\end{rmk}

\subsection*{Acknowledgement}

JP has been supported by a \textit{Visiting International Professor (VIP) Fellowship} from the Ruhr University Bochum. ZK and CT were supported by the DFG Scientific Network \textit{Cumulants, Concentration and Superconcentration}.

\bibliographystyle{plain}
\bibliography{raymond_bib}

\end{document}